%% file: BMMWiv1.tex
\documentclass[10pt]{amsart}

\usepackage{times,amsfonts,amsmath,amstext,amsbsy,amssymb,
  amsopn,amsthm,upref,eucal, amscd,color,url,caption}
\usepackage[T1]{fontenc}

\usepackage[pdftex]{graphicx}

\usepackage{tikz}
\usetikzlibrary{arrows}

\usepackage{longtable}
\usepackage{array}

\newtheorem{theorem}{Theorem}[section]
\newtheorem{lemma}[theorem]{Lemma}

\numberwithin{equation}{section}

\theoremstyle{definition}

\newtheorem{remark}[theorem]{Remark}

\def\leq{\leqslant }
\def\geq{\geqslant}

\begin{document}
\title[Rates of mixing for the Weil-Petersson geodesic flow I]{Rates of mixing for the Weil-Petersson geodesic flow I: \\ no rapid mixing in non-exceptional moduli spaces}

\author{Keith Burns}
\address{Keith Burns: Department of Mathematics, Northwestern University, 2033
Sheridan Road, Evanston, IL 60208-2730 USA.}
\email{burns@math.northwestern.edu.}
\urladdr{http://www.math.northwestern.edu/~burns/}

\author{Howard Masur}
\address{Howard Masur: Department of Mathematics, University of Chicago, 5734 S.
University, Chicago, IL 60607, USA.}
\email{masur@math.uchicago.edu.}
\urladdr{http://www.math.uchicago.edu/~masur/}

\author{Carlos Matheus}
\address{Carlos Matheus: Universit\'e Paris 13, Sorbonne Paris Cit\'e, LAGA, CNRS (UMR 7539), F-93439, Villetaneuse, France.}
\email{matheus@impa.br.}
\urladdr{http://www.impa.br/$\sim$cmateus}

\author{Amie Wilkinson}
\address{Amie Wilkinson: Department of Mathematics, University of Chicago, 5734 S.
University, Chicago, IL 60607, USA.}
\email{wilkinso@math.uchicago.edu.}
\urladdr{http://math.uchicago.edu/~wilkinso/}

\date{\today}

\begin{abstract}
We show that the rate of mixing of the Weil-Petersson flow on non-exceptional (higher dimensional) moduli spaces of Riemann surfaces is at most polynomial. 
\end{abstract}
\maketitle

\setcounter{tocdepth}{1}
\tableofcontents

\section{Introduction}\label{s.introduction}

The Weil-Petersson flow $\mathcal{WP}_t$ is the geodesic flow associated to the Weil-Petersson metric $g_{WP}$ on the moduli spaces $\mathcal{M}(S)=\mathcal{M}_{g,n}$ of Riemann surfaces $S$ of genus $g\geq 0$ and $n\geq 0$ punctures with $3g-3+n\geq 1$.

The Weil-Petersson (WP) metric is a natural object from the point of view of hyperbolic geometry of Riemann surfaces: for example, the WP metric has a simple expression in terms of the Fenchel-Nielsen coordinates (see Wolpert \cite{Wolpert1983}), the growth of hyperbolic lengths of simple closed curves in Riemann surfaces is related to the volumes of moduli spaces equipped with the WP metric (see Mirzakhani \cite{Mirzakhani}), and the thermodynamical invariants of the geodesics flows on Riemann surfaces allows  the  calculation of  the Weil-Petersson metric (see McMullen \cite{McMullen}).

The Weil-Petersson flow is a dynamically interesting object: indeed, the WP metric is an incomplete negatively curved, Riemannian (and, in fact, K\"ahler) metric; in particular, the WP flow is an example of a singular hyperbolic flow.

The dynamics of the WP flow has the following properties:

\begin{itemize}
\item Wolpert \cite{Wolpert2003} proved that the WP flow $\mathcal{WP}_t$ is defined for all times $t\in\mathbb{R}$ for almost every initial data with respect to the finite Liouville measure $\mu$ induced by the WP metric;
\item Brock, Masur and Minsky \cite{BMM} showed that the WP flow is transitive, its periodic orbits are dense and its topological entropy is infinite;
\item Burns, Masur and Wilkinson \cite{BMW} proved that the WP flow is ergodic and mixing (and even Bernoulli) non-uniformly hyperbolic flow whose metric entropy is positive and finite.
\end{itemize}

Intuitively, the ergodicity and mixing properties obtained by Burns, Masur and Wilkinson say that almost every WP flow orbit becomes equidistributed and all regions of the phase space are mixed together after a long time passes.

However, the arguments employed in \cite{BMW} -- notably involving Wolpert's curvature estimates for the WP metric, McMullen's primitive of the WP symplectic form, Katok and Strelcyn's work on Pesin theory for systems with singularities, and Hopf's argument for ergodicity of geodesic flows -- do not provide rates of mixing for the WP flow. In other words, even though the WP flow mixes together all regions of the phase space after a certain time, it is not clear how long it takes for such a mixture to occur. 

\subsection{Main results} \label{ss.results}

The goal of this paper is to prove that the rate of mixing of the WP flow is at most polynomial when the moduli space $\mathcal{M}_{g,n}$ is \emph{non-exceptional}, i.e., $3g-3+n>1$.

\begin{theorem}\label{t.BMMW-A} Consider the WP flow $\mathcal{WP}_t$ on the unit cotangent bundle of $\mathcal{M}_{g,n}$. Suppose that $3g-3+n>1$. Then, the rate of mixing of $\mathcal{WP}_t$ is at most polynomial: if the decay of correlations has the form
\begin{equation}\label{e.polymix}
C_t(f,g):=\left|\int f\cdot g\circ\mathcal{WP}_t - \int f \int g\right|\leq c\, t^{-\gamma}\|f\|_{C^1}\|g\|_{C^1} \quad\forall\, t\geq 1
\end{equation}
for some constants $c>0$ and $\gamma>0$, then $\gamma\leq 10$.
\end{theorem}

\begin{remark} Of course, the ``precise'' rate of mixing depends on the choice of functional spaces where the test functions live. Nevertheless, as it is explained in Remark \ref{r.poly-mixing-Ck} below, the arguments used in the proof of Theorem \ref{t.BMMW-A} can be adapted to show that, if we replace the $C^1$-norms by $C^{k+\alpha}$-norms, then one has an upper bound $\gamma\leq 2(k+\alpha)+8$ for the polynomial speed of decay of correlations of the WP flow. In other terms, our proof of Theorem \ref{t.BMMW-A} is sufficiently robust to show that the rate of mixing of the WP flow is not exponential (nor super-polynomial) even after ``natural'' modifications of the functional space of test functions.
\end{remark}

Our proof of Theorem \ref{t.BMMW-A} is based on the following geometric features of the WP metric. The volume of a $\varepsilon$-neighborhood of the  boundary of the completion of $\mathcal{M}_{g,n}$ is $\sim\varepsilon^4$, i.e., the cusp at infinity of $\mathcal{M}_{g,n}$ is polynomially large from the point of view of the WP metric.  We henceforth call this boundary the Cauchy boundary. 
Using this  fact, we construct a \emph{non-negligible} set (of volume $\sim\varepsilon^8$) of unit tangent vectors whose associated geodesics stay \emph{almost parallel} to certain parts of the Cauchy boundary of $\mathcal{M}_{g,n}$ for a \emph{non-trivial} fraction of time (of order $\sim 1/\varepsilon$). This implies that the WP flow takes a \emph{polynomially long} time to mix certain non-negligible regions near the Cauchy boundary of $\mathcal{M}_{g,n}$ with any given compact subset of $\mathcal{M}_{g,n}$, and, hence, the rate of mixing is not very fast in the sense that the exponent $\gamma$ appearing in \eqref{e.polymix} is not large (i.e., $\gamma\leq 10$).

Logically, the argument described in the previous paragraph breaks down in the \emph{exceptional} cases of the WP flow on the unit tangent bundle of the moduli spaces $\mathcal{M}_{g,n}$ with $3g-3+n=1$ because the Cauchy boundary of $\mathcal{M}_{g,n}$ with respect to the WP metric reduces to a finite collection of points whenever $3g-3+n=1$. Therefore, there are no WP geodesics travelling almost parallel to the boundary in the exceptional cases, and one might expect that the WP flow is exponentially mixing in these situations. 

In a forthcoming paper \cite{BMMWii}, we will show that this is indeed the case: the WP flow on the unit tangent bundle of $\mathcal{M}_{g,n}$ is exponentially mixing when $3g-3+n=1$.

\subsection{Questions and comments}  \label{comments}

Our Theorem \ref{t.BMMW-A} motivates the following question: is the WP flow on the unit tangent bundle of $\mathcal{M}_{g,n}$ \emph{genuinely} polynomially mixing when $3g-3+n>1$?

We contrast with the case of the Teichm\"uller flow where it is known that it is exponentially mixing \cite{AGY}.

We believe that this is a subtle question related to the geometry of the WP metric. Indeed, from the heuristic point of view, the \emph{exact} rate of mixing for the WP flow depends on how fast most WP flow orbits near the Cauchy boundary of $\mathcal{M}_{g,n}$ come back to compact parts of $\mathcal{M}_{g,n}$ (compare with our proof of Theorem \ref{t.BMMW-A}), and this is somewhat related to how close to zero the sectional curvatures of the WP metric can be.  Unfortunately, the best available expansion formulas (due to Wolpert) for the curvatures of WP metric near the Cauchy boundary of $\mathcal{M}_{g,n}$ do not help us here: for example there are sectional curvatures that are at least as close to zero as   
something of order $-e^{-\frac{1}{\epsilon}}$, where $\epsilon$ is the distance to the boundary of the  completion.
What then would be important is how large a measure set of planes have these curvatures close to $0$. 


We also mention a recent result of Hamenstadt \cite{H} connecting the  Teichm\"uller and Weil-Petersson flows.   She shows that  there is
a Borel subset
$E$ of the unit tangent bundle $Q$ of quadratic differentials which is invariant under the Teichm\"uller  flow and is 
of full measure
for every invariant Borel probability measure, and there is
a measurable map
$\lambda:E\to T^1\mathcal{T}_{g,n}$ to the unit tangent bundle for the Weil-Petersson metric 
 which conjugates  the Teichm\"uller flow 
into the Weil-Petersson flow. 
This conjugacy induces a continuous injection of the space of all 
Teichm\"uller invariant 
Borel probability measures on
$Q$ 
 into the space of all
WP 
invariant Borel
probability measures on 
$T^1\mathcal{T}_{g,n}$. Its image contains the Lebesgue Liouville measure.

\section{Proof of Theorem \ref{t.BMMW-A}} \label{sec:Thm:BMMW:A}

The basic references for the next subsection are Wolpert's papers \cite{Wolpert2008} and \cite{Wolpert2009} (see also Section 4 of \cite{BMW}).

\subsection{Background and notations}

Let $\mathcal{T}_{g,n}$ be the Teichm\"uller space of Riemann surfaces $M$ of genus $g\geq 0$ with $n\geq 0$ punctures. Recall that the fiber $T^{\ast}_M\mathcal{T}_{g,n}$ of the cotangent bundle $T^{\ast}\mathcal{T}_{g,n}$ of $\mathcal{T}_{g,n}$ can be identified with the space of integrable meromorphic quadratic differentials $\phi$ on $M$. In this context, the \emph{Weil-Petersson metric} is $$\langle\phi,\psi\rangle_{WP}=\int_M \phi\overline{\psi}(ds^2)^{-1}$$
where $ds^2$ is the hyperbolic metric associated to the underlying complex structure of $M$. The Weil-Petersson metric is $\Gamma_{g,n}$-invariant where $\Gamma_{g,n}$ is the mapping class group. In particular, the Weil-Petersson metric induces a metric on the moduli space $\mathcal{M}_{g,n} = \mathcal{T}_{g,n}/\Gamma_{g,n}$. The geodesic flow $\mathcal{WP}_t$ of the Weil-Petersson metric $g_{WP}:=\langle,\rangle_{WP}$ on the unit tangent bundle $T^1\mathcal{M}_{g,n} := T^1\mathcal{T}_{g,n}/\Gamma_{g,n}$ of the moduli space $\mathcal{M}_{g,n}$ is called \emph{Weil-Petersson flow}.

The Teichm\"uller space $\mathcal{T}_{g,n}$ is a complex manifold of dimension $3g-3+n$ homeomorphic to a ball. In terms of the identification of the cotangent bundle $T^*\mathcal{T}_{g,n}$ with the space of integrable meromorphic quadratic differentials, the natural almost complex structure $J$ on $\mathcal{T}_{g,n}$ simply corresponds to the multiplication by $i$. Also, a concrete way of constructing a homeomorphism between $\mathcal{T}_{g,n}$ and a ball of real dimension $6g-6+n$ uses the \emph{Fenchel-Nielsen coordinates}  
$$(\ell_{\alpha},\tau_{\alpha})_{\alpha\in P}:\mathcal{T}_{g,n}\to (\mathbb{R}_+\times\mathbb{R})^{3g-3+n}$$ 
obtained by fixing a pants decomposition $P=\{\alpha_1,\dots,\alpha_{3g-3+n}\}$ (maximal collection of non-peripheral, homotopically non-trivial simple closed curves with disjoint representatives) of $M$, and letting $\ell_{\alpha}(M)$ be the hyperbolic length of the unique geodesic of $M$ in the homotopy class of $\alpha$ and $\tau_{\alpha}(M)$ be a twist parameter. 

The WP metric is incomplete: it is possible to escape $\mathcal{T}_{g,n}$ in finite time by following a WP geodesic along which a simple closed curve $\alpha$ of hyperbolic length $\ell_{\alpha}$ shrinks into a point after time of order $\ell_{\alpha}^{1/2}$. The metric completion of $\mathcal{T}_{g,n}$ with respect to the WP metric is the \emph{augmented Teichm\"uller space} $\overline{\mathcal{T}}_{g,n}$ obtained by adjoining lower-dimensional Teichm\"uller spaces of noded Riemann surfaces. Alternatively, the Cauchy boundary $\partial \mathcal{T}_{g,n}$ of the Teichm\"uller space $\mathcal{T}_{g,n}$ equipped with the WP metric is 
$$\partial \mathcal{T}_{g,n} = \bigcup\limits_{\sigma}\mathcal{T}_{\sigma}$$
where $\mathcal{T}_{\sigma}$ is a Teichm\"uller space of noded Riemann surfaces where a certain collection $\sigma$ of non-peripheral, homotopically non-trivial, simple closed curves with disjoint representatives are pinched. 

The mapping class group $\Gamma_{g,n}$ acts on the augmented Teichm\"uller space $\overline{\mathcal{T}}_{g,n}$ and the quotient 
$$\overline{\mathcal{M}}_{g,n}:=\overline{\mathcal{T}}_{g,n}/\Gamma_{g,n}$$ 
is the so-called \emph{Deligne-Mumford compactification} of the moduli space $\mathcal{M}_{g,n}$. 

The picture below illustrates, for a given $\varepsilon>0$, the portion of $\mathcal{M}_{g,n}$ (with $3g-3+n>1$) consisting of $M\in\mathcal{M}_{g,n}$ such that a non-separating simple closed curve $\alpha$ has hyperbolic length $\ell_{\alpha}(M)\leq (2\varepsilon)^2$. 

\begin{figure}[htb!]
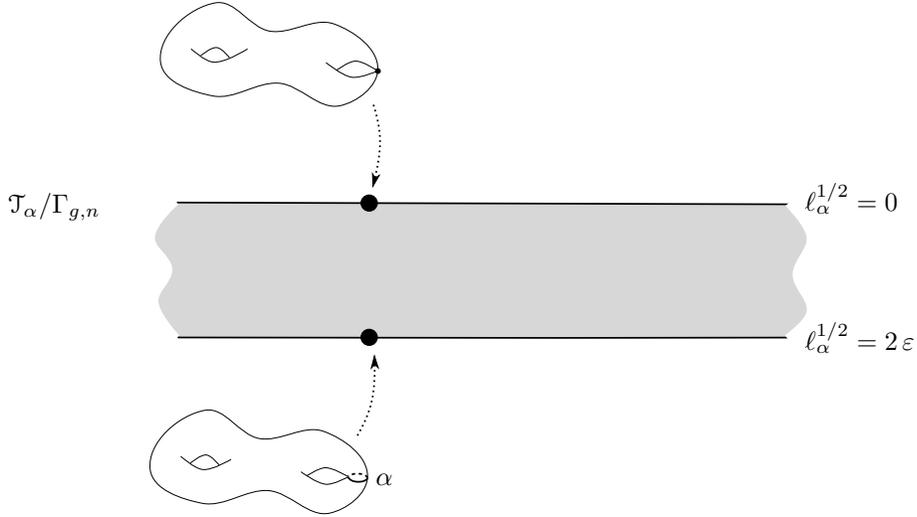
\caption{A portion of a neighborhood of $\partial\mathcal{M}_{g,n}$ when $3g-3+n>1$.}\label{f.BMMW1}
\end{figure}

\begin{remark}\label{r.non-exceptional-boundary} The stratum $\mathcal{T}_{\alpha}/\Gamma_{g,n}$ in Figure \ref{f.BMMW1} is non-trivial because $3g-3+n>1$. Indeed, by pinching $\alpha$ as above and by disconnecting the associated node, we get Riemann surfaces with genus $g-1$ and $n+2$ punctures whose moduli spaces are isomorphic to $\mathcal{T}_{\alpha}/\Gamma_{g,n}$. Thus, the complex orbifold $\mathcal{T}_{\alpha}/\Gamma_{g,n}$ has complex dimension $$3(g-1)-3+(n+2)=3g-3+n-1>0,$$ and, \emph{a fortiori}, it is non-trivial. 

Evidently, this argument breaks down for exceptional ($3g-3+n=1$) moduli spaces $\mathcal{M}_{g,n}$: for example, by pinching a curve $\alpha$ as above in an once-punctured torus and by removing the resulting node, we obtain thrice punctured spheres (whose moduli space is a single point), so that the boundary stratum of $\mathcal{M}_{1,1}$ reduces to a point. 

This difference between the structure of the boundaries of non-exceptional and exceptional moduli spaces partly explains the different rates of mixing of the WP flow on $\mathcal{M}_{g,n}$ depending on whether $3g-3+n>1$ or $3g-3+n=1$ (cf. Subsection \ref{ss.results}).
\end{remark}

The WP metric is a negatively curved K\"ahler metric, i.e., the sectional curvatures of the WP metric $g_{WP}$ are negative and the the symplectic Weil-Petersson $2$-form 
$$\omega_{WP}(v,Jw):=g_{WP}(v,w)$$ 
is closed. Furthermore, Wolpert showed that the symplectic WP form $\omega_{WP}$ admits the following  simple expression
$$\omega_{WP}=\frac{1}{2}\sum\limits_{\alpha\in P} d\ell_{\alpha}\wedge d\tau_{\alpha}$$
in terms of the Fenchel-Nielsen coordinates $(\ell_{\alpha}, \tau_{\alpha})_{\alpha\in P}$.

The geometry of the WP metric near $\partial{M}_{g,n}$ was described by Wolpert in terms of several asymptotic formulas. In the sequel, we recall some of Wolpert's formulas. 

\begin{lemma}\label{l.wp-DM-bdry-nbhd-vol} Denote by $\rho_0(M)$ the systole of $M$ and let 
$$E_{\rho}:=\{M\in\mathcal{M}_{g,n}: \rho_0(M)\leq\rho\}.$$ 
Then, the WP volume of $E_{\rho}$ is 
$$\textrm{vol}(E_{\rho})\simeq\rho^4$$
\end{lemma}

\begin{proof} Wolpert proved (in page 284 of \cite{Wolpert2008}) that $g_{WP}$ has asymptotic expansion
$$g_{WP}\sim \sum\limits_{\alpha\in\sigma} (4\, dx_{\alpha}^2 + x_{\alpha}^6 d\tau_{\alpha}^2)$$
near $\mathcal{T}_{\sigma}$, where $x_{\alpha}:=\ell_{\alpha}^{1/2}/\sqrt{2\pi^2}$.

This gives that the volume element $\sqrt{\det(g_{WP})}$ of the Weil-Petersson metric near $\mathcal{T}_{\sigma}$ is $\sim\prod\limits_{\alpha\in\sigma} x_{\alpha}^3$. Furthermore, this asymptotic expansion of $g_{WP}$ also says that the distance between $M$ and $\mathcal{T}_{\sigma}$ is comparable to $\min_{\alpha\in\sigma} x_{\alpha}(M)$. By putting these two facts together, we see that
$$\textrm{vol}(E_{\rho})\simeq \rho^4$$
This proves the lemma.
\end{proof}

Let $\mathcal{T}_{\sigma}$ be a boundary stratum of the augmented Teichm\"uller space $\overline{\mathcal{T}}_{g,n}$. We consider collections $\chi$ of simple closed curves such that each $\beta\in\chi$ is disjoint from all $\alpha\in\sigma$.  We let $\mathcal{B}$ be the set  of  $(\sigma,\chi)$. 

An element $(\sigma,\chi)\in\mathcal{B}$ is a \emph{combined length basis} at a point $M\in\mathcal{T}_{g,n}$ whenever the set of tangent vectors
$$\{\lambda_{\alpha}(M), J\lambda_{\alpha}(M), \textrm{grad}\,\ell_{\beta}(M)\}_{\alpha\in\sigma, \beta\in\chi}$$
is a basis of $T_M\mathcal{T}$, where $\ell_{\gamma}$ is the length parameter in the Fenchel-Nielsen coordinates, $J$ is the usual complex structure of $\mathcal{T}_{g,n}$, and $\lambda_{\alpha}:=\textrm{grad}\, \ell_{\alpha}^{1/2}$.

This notion can be extended to $\mathcal{T}_{\sigma}$ as follows. We say a collection $\chi$ is a \emph{relative basis} at a point $M_{\sigma}\in\mathcal{T}_{\sigma}$ whenever $(\sigma,\chi)\in\mathcal{B}$ and the length parameters $\{\ell_{\beta}\}_{\beta\in\chi}$ is a \emph{local} system of coordinates for $\mathcal{T}_{\sigma}$ near $X_{\sigma}$. In other words, $\chi$ is a relative basis at $X_{\sigma}\in\mathcal{T}_{\sigma}$ if and only if $\{\textrm{grad} \ell_{\beta}\}_{\beta\in\chi}$ is a basis of $T_{X_{\sigma}}\mathcal{T}_{\sigma}$.

In this setting, Wolpert proved the following results:

\begin{theorem}[Wolpert \cite{Wolpert2008}]\label{t.Wolpert2008} For any point $M_{\sigma}\in\mathcal{T}_{\sigma}$, there exists a relative length basis $\chi$. Furthermore, the WP metric $\langle.,.\rangle_{WP}$ can be written as
$$\langle.,.\rangle_{WP}\sim \sum\limits_{\alpha\in\sigma}\left((d\ell_{\alpha}^{1/2})^2+(d\ell_{\alpha}^{1/2}\circ J)^2\right) + \sum\limits_{\beta\in\chi}(d\ell_{\beta})^2$$
where the implied comparison constant is uniform in a neighborhood $U\subset\overline{\mathcal{T}}$ of $X_{\sigma}$.

In particular, there exists a neighborhood $V\subset\overline{\mathcal{T}}$ of $X_{\sigma}$ such that $(\sigma,\chi)$ is a combined length basis at any $X\in V\cap\mathcal{T}$.
\end{theorem}

Before writing down the next formula for the WP metric, we need the following notations. Given $\mu$ an arbitrary collection of simple closed curves on a surface $S$ of genus $g$ with $n$ punctures, we define
$$\underline{\ell}_{\mu}(X):=\min\limits_{\alpha\in\mu}\ell_{\alpha}(X)\quad \textrm{and}\quad \overline{\ell}_{\mu}(X):=\max\limits_{\alpha\in\mu}\ell_{\alpha}(X)$$
where $X\in\mathcal{T}_{g,n}$. Also, given a constant $c>1$ and a basis $(\sigma,\chi)\in\mathcal{B}$, we will consider the following region of Teichm\"uller space:
$$\Omega(\sigma,\chi,c):=\{X\in\mathcal{T}: 1/c<\underline{\ell}_{\chi}(X) \textrm{ and } \overline{\ell}_{\sigma\cup\chi}(X)<c\}$$

Wolpert \cite{Wolpert2009} proved several (uniform) estimates (on the regions $\Omega(\sigma,\chi,c)$) for the WP metric (and its sectional curvatures) in terms of the basis $\lambda_{\alpha}=\textrm{grad}\ell_{\alpha}^{1/2}$, $\alpha\in\sigma$ and $\textrm{grad}\ell_{\beta}$, $\beta\in\chi$. For our purposes, we will need just the following asymptotic formula: 

\begin{theorem}[Wolpert \cite{Wolpert2009}]\label{t.Wolpert-expansions} Fix $c>1$. Then, for any $(\sigma,\chi)\in\mathcal{B}$, and any $\alpha,\alpha'\in\sigma$ and $\beta,\beta'\in\chi$, the following estimate hold uniformly on $\Omega(\sigma,\chi,c)$: for any vector $v\in T\Omega(\sigma,\chi,c)$, we have 
$$\left\|\nabla_v \lambda_{\alpha} - \frac{3}{2\pi\ell_{\alpha}^{1/2}}\langle v, J\lambda_{\alpha}\rangle J\lambda_{\alpha}\right\|_{WP} =
O(\ell_{\alpha}^{3/2}\|v\|_{WP})$$
\end{theorem}

As it was observed in Lemma 4.15 of \cite{BMW}, this theorem implies the following Clairaut-type relation.
Given $v\in T^1\mathcal{M}_{g,n}$, let us consider the quantity $$r_{\alpha}(v):=\sqrt{\langle v, \lambda_{\alpha} \rangle^2 + \langle v, J\lambda_{\alpha}\rangle^2}$$ and  the complex line $$L=\textrm{span}\{\lambda_{\alpha}, J\lambda_{\alpha}\}.$$ By definition, $r_{\alpha}(v)$ measures the size of the projection of the unit vector $v$ in the complex line $L$. In particular, we can think of $v$ as ``almost parallel'' to $\mathcal{T}_{\alpha}/\Gamma_{g,n}$ whenever the quantity $r_{\alpha}(v)$ is very close to zero.

\begin{lemma}\label{l.Clairaut-ODE} Fix $c>1$. For every $(\sigma,\chi)\in\mathcal{B}$, $\alpha\in\sigma$,  $\gamma:I\to\mathcal{T}_{g,n}$ a WP-geodesic segment, and $t\in I$ with $\gamma(t)\in\Omega(\sigma,\chi,c)$, we have  
$$r_{\alpha}'(t) = O (f_{\alpha}(t)^3)$$
where $r_{\alpha}(t):=\sqrt{\langle \lambda_{\alpha}, \dot{\gamma}(t)\rangle^2 + \langle J\lambda_{\alpha}, \dot{\gamma}(t)\rangle^2}$ and $f_{\alpha}(t):=\sqrt{\ell_{\alpha}(\gamma(t))}$. 
\end{lemma}

\begin{proof} By differentiating $r_{\alpha}(t)^2=\langle \lambda_{\alpha}, \dot{\gamma}(t)\rangle^2 + \langle J\lambda_{\alpha}, \dot{\gamma}(t)\rangle^2$, we see that
$$2r_{\alpha}(t)r_{\alpha}'(t) = 2 \langle \lambda_{\alpha}, \dot{\gamma}(t)\rangle \langle \nabla_{\dot{\gamma}(t)}\lambda_{\alpha}, \dot{\gamma}(t)\rangle + 2 \langle J\lambda_{\alpha}, \dot{\gamma}(t)\rangle \langle J\nabla_{\dot{\gamma}(t)}\lambda_{\alpha}, \dot{\gamma}(t)\rangle.$$
Here, we used the fact that the WP metric is K\"ahler, so that $J$ is parallel. 

Now, we observe that, by Wolpert's formulas (cf. Theorem \ref{t.Wolpert-expansions}), one can write $\nabla_{\dot{\gamma}(t)}\lambda_{\alpha}$ and $J\nabla_{\dot{\gamma}(t)}\lambda_{\alpha}$ as 
$$\nabla_{\dot{\gamma}(t)}\lambda_{\alpha}=\frac{3\langle\dot{\gamma}(t), J\lambda_{\alpha}\rangle}{2\pi f_{\alpha}(t)} J\lambda_{\alpha} + O(f_{\alpha}(t)^3)$$
and
$$J\nabla_{\dot{\gamma}(t)}\lambda_{\alpha}=-\frac{3\langle\dot{\gamma}(t), J\lambda_{\alpha}\rangle}{2\pi f_{\alpha}(t)} \lambda_{\alpha} + O(f_{\alpha}(t)^3)$$

Since $\max\{|\langle\dot{\gamma}(t), \lambda_{\alpha}|, |\langle\dot{\gamma}(t), J\lambda_{\alpha}\rangle|\}\leq r_{\alpha}(t)$ (by definition), we conclude from the previous equations that
\begin{eqnarray*}
2r_{\alpha}(t)r_{\alpha}'(t) &=& \frac{3}{\pi f_{\alpha}(t)}
(\langle \lambda_{\alpha}, \dot{\gamma}(t)\rangle \langle J\lambda_{\alpha}, \dot{\gamma}(t)\rangle^2 - \langle \lambda_{\alpha}, \dot{\gamma}(t)\rangle \langle J\lambda_{\alpha}, \dot{\gamma}(t)\rangle^2) \\ &+& O(r_{\alpha}(t) f_{\alpha}(t)^3) \\
&=& 0 + O(r_{\alpha}(t)f_{\alpha}(t)^3).
\end{eqnarray*}
This proves the lemma.
\end{proof}

\subsection{Construction of certain long segments of WP geodesics close to $\partial\mathcal{M}_{g,n}$}

After these preliminaries on the geometry of the WP metric, we are ready to begin the proof of Theorem \ref{t.BMMW-A}. 

Fix $\alpha$ a non-separating (non-peripheral, homotopically non-trivial) simple closed curve. We want to locate certain regions near $\mathcal{T}_{\alpha}/\Gamma_{g,n}$ taking a long time to mix with the compact part of $\mathcal{M}_{g,n}$. For this sake, we will exploit Lemma \ref{l.Clairaut-ODE} to build nice sets of unit vectors traveling in an ``almost parallel'' way to $\mathcal{T}_{\alpha}/\Gamma_{g,n}$ for a significant amount of time. 

More precisely, recall  the complex line $L=\textrm{span}\{\lambda_{\alpha}, J\lambda_{\alpha}\}$. Intuitively, $L$  points in the normal direction to a ``copy'' of $\mathcal{T}_{\alpha}/\Gamma_{g,n}$ inside a level set of the function $\ell_{\alpha}^{1/2}$ as indicated below:

\begin{figure}[htb!]
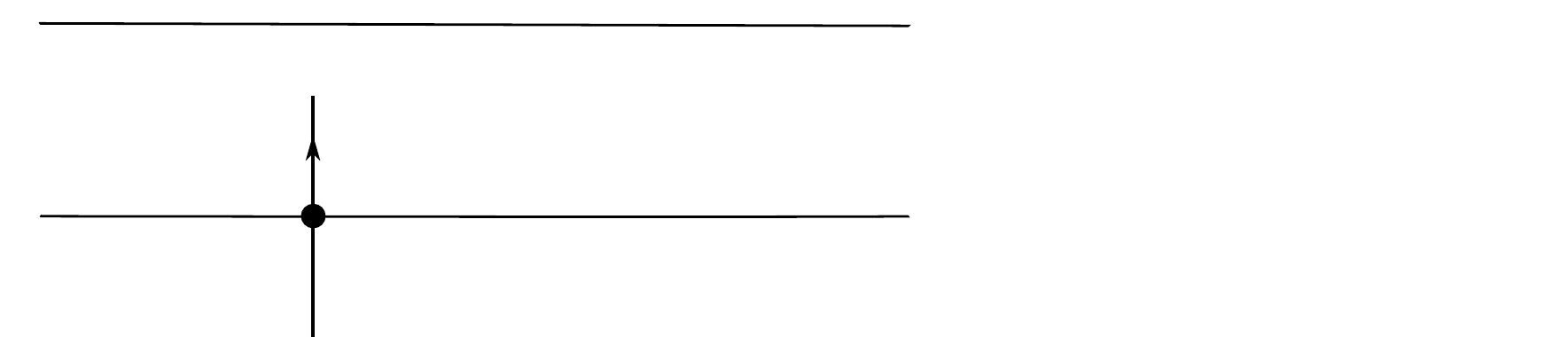
\end{figure}

Using $L$, we can formalize the notion of ``almost parallel'' vector to $\mathcal{T}_{\alpha}/\Gamma_{g,n}$.
We will show that unit vectors almost parallel to $\mathcal{T}_{\alpha}/\Gamma_{g,n}$ whose footprints are close to $\mathcal{T}_{\alpha}/\Gamma_{g,n}$ always generate geodesics staying near $\mathcal{T}_{\alpha}/\Gamma_{g,n}$ for a long time. More concretely, given $\varepsilon>0$, let us define the set 
$$V_{\varepsilon} := \{v\in T^1\mathcal{M}_{g,n}: f_{\alpha}(v)\leq \varepsilon, \, r_{\alpha}(v)\leq \varepsilon^2\}$$
where $f_{\alpha}(v) := \ell_{\alpha}^{1/2}(p)$ and $p\in\mathcal{M}_{g,n}$ is the footprint of the unit vector $v\in T^1\mathcal{M}_{g,n}$. Equivalently, $V_{\varepsilon}$ is the disjoint union of the pieces of spheres $S_{\varepsilon}(p):=\{v\in T^1_p\mathcal{M}_{g,n}: r_{\alpha}(v)\leq \varepsilon^2\}$ attached to points $p\in\mathcal{M}_{g,n}$ with $\ell_{\alpha}(p)\leq \varepsilon^2$. The following figure summarizes the geometry of $S_{\varepsilon}(p)$:

\begin{figure}[htb!]
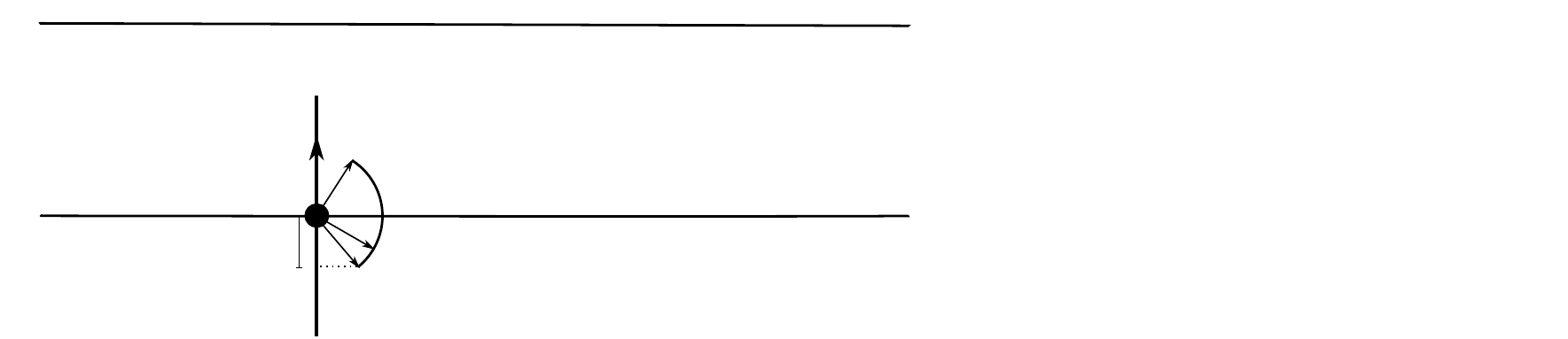
\end{figure}

In this setting, we will show that a geodesic $\gamma_v(t)$ originating at any $v\in V_{\varepsilon}$ stays in a $(2\varepsilon)^2$-neighborhood of $\mathcal{T}_{\alpha}/\Gamma_{g,n}$ for an interval of  time $[0, T]$ of size of order $1/\varepsilon$. 

\begin{figure}[htb!]
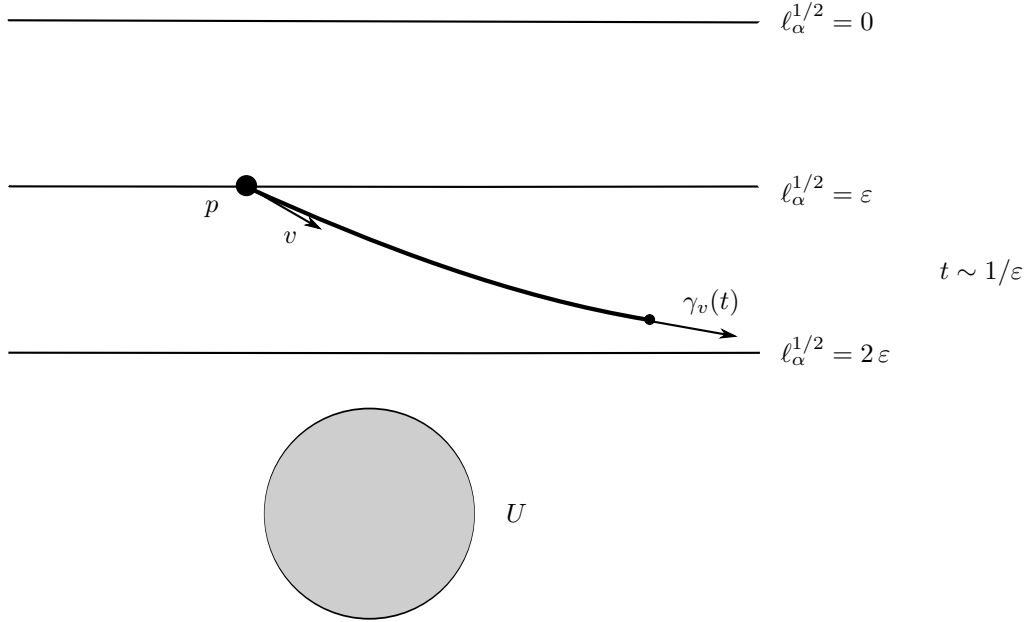
\caption{WP geodesic segment travelling almost parallel to $\partial \mathcal{M}_{g,n}$.}\label{f.BMMW2}
\end{figure}

\begin{lemma}\label{l.WP-mixing-time-lower-bound} There exists a constant $C_0>0$ (depending only on $g$ and $n$) such that 
$$\sqrt{\ell_{\alpha}(\gamma_v(t))}=f_{\alpha}(t)\leq 2\varepsilon$$
for all $v\in V_{\varepsilon}$ and $0\leq t\leq 1/C_0\varepsilon$. 
\end{lemma}

\begin{proof} By definition, $v\in V_{\varepsilon}$ implies that $f_{\alpha}(0)\leq \varepsilon$. Thus, it makes sense to consider the maximal interval $[0,T]$ of time such that $f_{\alpha}(t)\leq 2\varepsilon$ for all $0\leq t\leq T$. 

By Lemma \ref{l.Clairaut-ODE}, we have that $r_{\alpha}'(s)=O(f_{\alpha}(s)^3)$, i.e., $|r_{\alpha}'(s)|\leq B\cdot f_{\alpha}(s)^3$ for some constant $B>1/4$ depending only on $g$ and $n$. In particular, $|r_{\alpha}'(s)|\leq B\cdot f_{\alpha}(s)^3\leq B\cdot (2\varepsilon)^3$ for all $0\leq s\leq T$. From this estimate, we deduce that 
$$r_{\alpha}(t) = r_{\alpha}(0)+\int_0^t r_{\alpha}'(s)\,ds\leq r_{\alpha}(0)+B\cdot (2\varepsilon)^3 t = r_{\alpha}(0)+8B\varepsilon^3t$$
for all $0\leq t\leq T$. Since the fact that $v\in V_{\varepsilon}$ implies that $r_{\alpha}(0)\leq\varepsilon^2$, the previous inequality tell us that 
$$r_{\alpha}(t)\leq \varepsilon^2+8B\varepsilon^3t$$
for all $0\leq t\leq T$. 

Next, we observe that, by definition, $f_{\alpha}'(t)=\langle\dot{\gamma}(t),\textrm{grad}\ell_{\alpha}^{1/2}\rangle = \langle\dot{\gamma}(t),\lambda_{\alpha}\rangle$. Hence, 
$$|f_{\alpha}'(t)|=|\langle\dot{\gamma}(t),\lambda_{\alpha}\rangle|\leq \sqrt{\langle\dot{\gamma}(t),\lambda_{\alpha}\rangle^2 + \langle\dot{\gamma}(t),J\lambda_{\alpha}\rangle^2} = r_{\alpha}(t)$$ 
By putting together the previous two inequalities and the fact that $f_{\alpha}(0)\leq\varepsilon$ (as $v\in V_{\varepsilon}$), we conclude that 
$$f_{\alpha}(T) = f_{\alpha}(0)+\int_0^T f_{\alpha}'(t) \, dt\leq \varepsilon + \varepsilon^2 T + 4B\varepsilon^3T^2$$
Since $T>0$ was chosen so that $[0,T]$ is the maximal interval with $f_{\alpha}(t)\leq 2\varepsilon$ for all $0\leq t\leq T$, we have that $f_{\alpha}(T)=2\varepsilon$. Therefore, the previous estimate can be rewritten as 
$$2\varepsilon\leq \varepsilon + \varepsilon^2T + 4B\varepsilon^3T^2$$
Because $B>1/4$, it follows from this inequality that $T\geq 1/C_0\varepsilon$ where $C_0:=8B$. 

In other words, we showed that $[0,1/C_0\varepsilon]\subset [0,T]$, and, \emph{a fortiori}, $f_{\alpha}(t)\leq 2\varepsilon$ for all $0\leq t\leq 1/C_0\varepsilon$. This completes the proof of the lemma.
\end{proof}

Geometrically, this lemma says that the WP geodesic flow does \emph{not} mix $V_{\varepsilon}$ with any fixed ball $U$ in the compact part of $\mathcal{M}_{g,n}$ of Riemann surfaces with systole $> (2\varepsilon)^2$: see Figure \ref{f.BMMW2} above. 

Once we have Lemma \ref{l.WP-mixing-time-lower-bound} in our toolbox, it is not hard to infer some upper bounds on the rate of mixing of the WP flow on $T^1\mathcal{M}_{g,n}$ when $3g-3+n>1$. 

\subsection{End of proof of Theorem \ref{t.BMMW-A}}

Suppose that the WP flow $\mathcal{WP}_t$ on $T^1\mathcal{M}_{g,n}$ has a rate of mixing of the form 
$$C_t(a,b) = \left|\int a\cdot b\circ\mathcal{WP}_t - \left(\int a\right)\left(\int b\right)\right|\leq C t^{-\gamma}\|a\|_{C^1}\|b\|_{C^1}$$
for some constants $C>0$, $\gamma>0$, for all $t\geq 1$, and for all choices of $C^1$-observables $a$ and $b$.

Let us fix once and for all an open ball $U$ (with respect to the WP metric) contained in the compact part of $\mathcal{M}_{g,n}$: this means that there 
exists $\varepsilon_0>0$ such that the systoles of all Riemann surfaces in $U$ are $\geq\varepsilon_0^2$. 

Take a $C^1$ function $a$ supported on the set $T^1U$ of unit vectors with footprints on $U$ with values $0\leq a\leq 1$ such that $\int a\geq \textrm{vol}(U)/2$ and $\|a\|_{C^1}=O(1)$: such a function $a$ can be easily constructed by smoothing the characteristic function of $U$ with the aid of bump functions. 
  
Next, for each $\varepsilon>0$, take a $C^1$ function $b_{\varepsilon}$ supported on the set $V_{\varepsilon}$ with values $0\leq b_{\varepsilon}\leq 1$ such that $\int b_{\varepsilon}\geq \textrm{vol}(V_{\varepsilon})/2$ and $\|b_{\varepsilon}\|_{C^1}=O(1/\varepsilon^2)$: such a function $b_{\varepsilon}$ can also be constructed by smoothing the characteristic function of $V_{\varepsilon}$ after taking into account the description of the WP metric near $\mathcal{T}_{\alpha}/\Gamma_{g,n}$ given by Theorem \ref{t.Wolpert2008} above and the definition of $V_{\varepsilon}$ (in terms of the conditions $\ell_{\alpha}^{1/2}\leq \varepsilon$ and $r_{\alpha}\leq\varepsilon^2$). Furthermore, this description of the WP metric $g_{WP}$ near $\mathcal{T}_{\alpha}/\Gamma_{g,n}$ combined with the asymptotic expansion $g_{WP}\sim 4dx_{\alpha}^2+x_{\alpha}^6d\tau_{\alpha}$ where $x_{\alpha}:=\ell_{\alpha}^{1/2}/\sqrt{2\pi^2}$ and $\tau_{\alpha}$ is a twist parameter (see the proof of Lemma \ref{l.wp-DM-bdry-nbhd-vol} above) says that $\textrm{vol}(V_{\varepsilon})\sim \varepsilon^8$: indeed, the condition $f_{\alpha}=\ell_{\alpha}^{1/2}\leq\varepsilon$ on footprints of unit tangent vectors in $V_{\varepsilon}$ provides a set of volume $\sim \varepsilon^4$ and the condition $r_{\alpha}\leq\varepsilon^2$ on unit tangent vectors in $V_{\varepsilon}$ with a fixed footprint provides a set of volume comparable to the Euclidean area $\pi\varepsilon^4$ of the Euclidean ball $\{\vec{v}\in\mathbb{R}^2: |v|\leq\varepsilon^2\}$ (cf. Theorem \ref{t.Wolpert2008}), so that 
$$\textrm{vol}(V_{\varepsilon}) = \int_{\{\ell_{\alpha}^{1/2}(p)\leq\varepsilon\}} \textrm{vol}(\{v\in T^1_p\mathcal{M}_{g,n}: r_{\alpha}(v)\leq\varepsilon^2\})\sim (\pi\varepsilon^4)\cdot \varepsilon^4\sim \varepsilon^8$$ 

In summary, for each $\varepsilon>0$, we have a $C^1$ function $b_{\varepsilon}$ supported on $V_{\varepsilon}$ with $0\leq b\leq 1$, $\|b_{\varepsilon}\|_{C^1} = O(1/\varepsilon^2)$ and $\int b_{\varepsilon}\geq c_0\varepsilon^8$ for some constant $c_0>0$ depending only on $g$ and $n$. 

Given $\gamma>10$, our plan is to use the observables $a$ and $b_{\varepsilon}$ to contradict the inequality
$$C_{t}(a,b_{\varepsilon})=\left|\int a\cdot b_{\varepsilon}\circ\mathcal{WP}_t - \left(\int a\right)\left(\int b_{\varepsilon}\right)\right|\leq 
C t^{-\gamma}\|a\|_{C^1}\|b_{\varepsilon}\|_{C^1}$$ 
for $t\sim 1/\varepsilon$ and $\varepsilon>0$ sufficiently small.  

In this direction, we apply Lemma \ref{l.WP-mixing-time-lower-bound} to obtain a constant $C_0>0$ such that $V_{\varepsilon}\cap \mathcal{WP}_{\frac{1}{C_0\varepsilon}}(T^1U)=\emptyset$ whenever $2\varepsilon<\varepsilon_0$. Indeed, since $V_{\varepsilon}$ is a symmetric set (i.e., $v\in V_{\varepsilon}$ if and only if $-v\in V_{\varepsilon}$), it follows from Lemma \ref{l.WP-mixing-time-lower-bound} that all Riemann surfaces in the footprints of $\mathcal{WP}_{-\frac{1}{C_0\varepsilon}}(V_{\varepsilon})$ have a systole $\leq (2\varepsilon)^2<\varepsilon_0^2$. Because we took $U$ in such a way that all Riemann surfaces in $U$ have systole $\geq \varepsilon_0^2$, we obtain $\mathcal{WP}_{-\frac{1}{C_0\varepsilon}}(V_{\varepsilon})\cap T^1U=\emptyset$, that is, $V_{\varepsilon}\cap \mathcal{WP}_{\frac{1}{C_0\varepsilon}}(T^1U)=\emptyset$, as it was claimed. 

Now, we observe that the function $a\cdot b_{\varepsilon}\circ \mathcal{WP}_t$ is supported on $\mathcal{WP}_{-t}(V_{\varepsilon})\cap T^1U$ because $a$ is supported on $T^1U$ and $b_{\varepsilon}$ is supported on $V_{\varepsilon}$. It follows from the claim in the previous paragraph, we deduce that $a\cdot b_{\varepsilon}\circ \mathcal{WP}_{\frac{1}{C_0\varepsilon}}\equiv 0$ whenever $2\varepsilon<\varepsilon_0$. Thus,   
$$C_{\frac{1}{C_0\varepsilon}}(a,b_{\varepsilon}) := \left|\int a\cdot b_{\varepsilon}\circ\mathcal{WP}_{\frac{1}{C_0\varepsilon}} - \left(\int a\right)\left(\int b_{\varepsilon}\right)\right| = 
\left(\int a\right)\left(\int b_{\varepsilon}\right)$$ 

By plugging this identity into the polynomial decay of correlations estimate $C_t(a,b_{\varepsilon})\leq Ct^{-\gamma}\|a\|_{C^1}\|b_{\varepsilon}\|_{C^1}$, we get 
$$\left(\int a\right)\left(\int b_{\varepsilon}\right) = C_{\frac{1}{C_0\varepsilon}}(a,b_{\varepsilon})\leq CC_0^{\gamma}\varepsilon^{\gamma}\|a\|_{C^1}\|b_{\varepsilon}\|_{C^1}$$
whenever $2\varepsilon<\varepsilon_0$ and $1/C_0\varepsilon\geq 1$. 

We affirm that the previous estimate implies that $\gamma\leq 10$. In fact, recall that our choices were made so that $\int a\geq \textrm{vol}(U)/2$ where $U$ is a fixed ball, $\|a\|_{C^1}=O(1)$, $\int b_{\varepsilon} \geq c_0\varepsilon^8$ for some constant $c_0>0$ and $\|b_{\varepsilon}\|_{C^1}=O(1/\varepsilon^2)$. Hence, by combining these facts and the previous mixing rate estimate, we get that 
$$\left(\frac{\textrm{vol}(U)}{2}\right) c_0\varepsilon^8\leq \left(\int a\right)\left(\int b_{\varepsilon}\right)\leq CC_0^{\gamma}\varepsilon^{\gamma}\|a\|_{C^1}\|b_{\varepsilon}\|_{C^1} = O(\varepsilon^{\gamma}\frac{1}{\varepsilon^2}),$$
that is, $\varepsilon^{10}\leq D \varepsilon^{\gamma}$, for some constant $D>0$ and for all $\varepsilon>0$ sufficiently small (so that $2\varepsilon<\varepsilon_0$ and $1/C_0\varepsilon\geq 1$). It follows that $\gamma\leq 10$, as we claimed. This completes the proof of Theorem \ref{t.BMMW-A}.

\begin{remark}\label{r.poly-mixing-Ck} In the statement of the previous proposition, the choice of $C^1$-norms to measure the rate of mixing of the WP flow is not very important. Indeed, an inspection of the construction of the functions $b_{\varepsilon}$ in the argument above reveals that $\|b_{\varepsilon}\|_{C^{k+\alpha}} = O(1/\varepsilon^{k+\alpha})$ for any $k\in\mathbb{N}$, $0\leq\alpha<1$. In particular, the proof of the previous proposition is sufficiently robust to show also that a rate of mixing of the form
$$C_t(a,b) = \left|\int a\cdot b\circ\varphi_t - \left(\int a\right)\left(\int b\right)\right|\leq C t^{-\gamma}\|a\|_{C^{k+\alpha}}\|b\|_{C^{k+\alpha}}$$
for some constants $C>0$, $\gamma>0$, for all $t\geq 1$, and for all choices of $C^1$-observables $a$ and $b$ holds only if $\gamma\leq 8+2(k+\alpha)$. 

In other words, even if we replace $C^1$-norms by (stronger, smoother) $C^{k+\alpha}$-norms in our measurements of rates of mixing of the WP flow (on $T^1\mathcal{M}_{g,n}$ for $3g-3+n>1$), our discussions so far will always give polynomial upper bounds for the decay of correlations. 
\end{remark}

We finish with the following result which   answers a question of F. Ledrappier to the authors.
In spite of the fact that the metric is not complete we have 
\begin{theorem} The WP metric is stochastically complete, i.e., the lifetime of the corresponding Brownian motion associated to the Laplace-Beltrami operator of the WP metric is infinite.
\end{theorem}
\begin{proof}
 This follows from the following argument. By Lemma~\ref{l.wp-DM-bdry-nbhd-vol}  the volume of a $\varepsilon$-neighborhood $N_{\varepsilon}$ of the Cauchy boundary $\partial \mathcal{M}(S)$ of the moduli space $\mathcal{M}(S)$ equipped with the WP metric is $\sim\varepsilon^4$. Therefore, the lower Minkowski codimension $\liminf\limits_{\varepsilon\to 0}\frac{\log\textrm{vol}(N_{\varepsilon})}{\log\varepsilon}=4$ is greater than $2$ and, by Masamune \cite{Masamune}, the Cauchy boundary $\partial \mathcal{M}(S)$ is an almost polar set. By combining this with the facts that the total volume of $\mathcal{M}(S)$ is finite and the Cauchy boundary $\partial \mathcal{M}(S)$ is bounded, we deduce from the results of Grigoryan and Masamune \cite{GrigoryanMasamune} that the WP metric is stochastically complete. 
\end{proof}



\end{document}

%% file: Fig1-WP-2013-6.pdf_tex
\begingroup%
  \makeatletter%
  \providecommand\color[2][]{%
    \errmessage{(Inkscape) Color is used for the text in Inkscape, but the package 'color.sty' is not loaded}%
    \renewcommand\color[2][]{}%
  }%
  \providecommand\transparent[1]{%
    \errmessage{(Inkscape) Transparency is used (non-zero) for the text in Inkscape, but the package 'transparent.sty' is not loaded}%
    \renewcommand\transparent[1]{}%
  }%
  \providecommand\rotatebox[2]{#2}%
  \ifx\svgwidth\undefined%
    \setlength{\unitlength}{390bp}%
    \ifx\svgscale\undefined%
      \relax%
    \else%
      \setlength{\unitlength}{\unitlength * \real{\svgscale}}%
    \fi%
  \else%
    \setlength{\unitlength}{\svgwidth}%
  \fi%
  \global\let\svgwidth\undefined%
  \global\let\svgscale\undefined%
  \makeatother%
  \begin{picture}(0.8,0.49495473)%
    \put(0,0){\includegraphics[width=\unitlength]{Fig1-WP-2013-6.pdf}}%
    \put(0.33309733,0.02799569){\color[rgb]{0,0,0}\makebox(0,0)[lb]{\smash{$\alpha$}}}%
    \put(0.74805664,0.15882903){\color[rgb]{0,0,0}\makebox(0,0)[lb]{\smash{$\ell^{1/2}_{\alpha}=2\,\varepsilon$}}}%
    \put(0.74805664,0.29359303){\color[rgb]{0,0,0}\makebox(0,0)[lb]{\smash{$\ell^{1/2}_{\alpha}=0$}}}%
    \put(-0.02206933,0.29116236){\color[rgb]{0,0,0}\makebox(0,0)[lb]{\smash{$\mathcal{T}_{\alpha}/\Gamma_{g,n}$}}}%
  \end{picture}%
\endgroup%

%% file: Fig2-WP-2013-6.pdf_tex
\begingroup%
  \makeatletter%
  \providecommand\color[2][]{%
    \errmessage{(Inkscape) Color is used for the text in Inkscape, but the package 'color.sty' is not loaded}%
    \renewcommand\color[2][]{}%
  }%
  \providecommand\transparent[1]{%
    \errmessage{(Inkscape) Transparency is used (non-zero) for the text in Inkscape, but the package 'transparent.sty' is not loaded}%
    \renewcommand\transparent[1]{}%
  }%
  \providecommand\rotatebox[2]{#2}%
  \ifx\svgwidth\undefined%
    \setlength{\unitlength}{509.20205078bp}%
    \ifx\svgscale\undefined%
      \relax%
    \else%
      \setlength{\unitlength}{\unitlength * \real{\svgscale}}%
    \fi%
  \else%
    \setlength{\unitlength}{\svgwidth}%
  \fi%
  \global\let\svgwidth\undefined%
  \global\let\svgscale\undefined%
  \makeatother%
  \begin{picture}(1,0.21489352)%
    \put(0,0){\includegraphics[width=\unitlength]{Fig2-WP-2013-6.pdf}}%
    \put(0.59708559,0.06575086){\color[rgb]{0,0,0}\makebox(0,0)[lb]{\smash{$\ell^{1/2}_{\alpha}=\varepsilon$}}}%
    \put(0.59708559,0.19278632){\color[rgb]{0,0,0}\makebox(0,0)[lb]{\smash{$\ell^{1/2}_{\alpha}=0$}}}%
    \put(0.2127582,0.12518802){\color[rgb]{0,0,0}\makebox(0,0)[lb]{\smash{$L$}}}%
    \put(0.1709,0.04102266){\color[rgb]{0,0,0}\makebox(0,0)[lb]{\smash{$p$}}}%
  \end{picture}%
\endgroup%

%% file: Fig3-WP-2013-6.pdf_tex
\begingroup%
  \makeatletter%
  \providecommand\color[2][]{%
    \errmessage{(Inkscape) Color is used for the text in Inkscape, but the package 'color.sty' is not loaded}%
    \renewcommand\color[2][]{}%
  }%
  \providecommand\transparent[1]{%
    \errmessage{(Inkscape) Transparency is used (non-zero) for the text in Inkscape, but the package 'transparent.sty' is not loaded}%
    \renewcommand\transparent[1]{}%
  }%
  \providecommand\rotatebox[2]{#2}%
  \ifx\svgwidth\undefined%
    \setlength{\unitlength}{509.20454102bp}%
    \ifx\svgscale\undefined%
      \relax%
    \else%
      \setlength{\unitlength}{\unitlength * \real{\svgscale}}%
    \fi%
  \else%
    \setlength{\unitlength}{\svgwidth}%
  \fi%
  \global\let\svgwidth\undefined%
  \global\let\svgscale\undefined%
  \makeatother%
  \begin{picture}(1,0.21608585)%
    \put(0,0){\includegraphics[width=\unitlength]{Fig3-WP-2013-6.pdf}}%
    \put(0.59708754,0.06694391){\color[rgb]{0,0,0}\makebox(0,0)[lb]{\smash{$\ell^{1/2}_{\alpha}=\varepsilon$}}}%
    \put(0.59708754,0.19397876){\color[rgb]{0,0,0}\makebox(0,0)[lb]{\smash{$\ell^{1/2}_{\alpha}=0$}}}%
    \put(0.20687047,0.15819512){\color[rgb]{0,0,0}\makebox(0,0)[lb]{\smash{$L$}}}%
    \put(0.17780553,0.08804765){\color[rgb]{0,0,0}\makebox(0,0)[lb]{\smash{$p$}}}%
    \put(0.17153246,0.05686539){\scriptsize\color[rgb]{0,0,0}\makebox(0,0)[lb]{\smash{$\varepsilon^2$}}}%
    \put(0.24123755,0.10052061){\color[rgb]{0,0,0}\makebox(0,0)[lb]{\smash{$S_{\varepsilon}(p)$}}}%
    \put(0.22123755,0.03018363){\color[rgb]{0,0,0}\makebox(0,0)[lb]{\smash{$v$}}}%
  \end{picture}%
\endgroup%

%% file: Fig4-WP-2013-6.pdf_tex
\begingroup%
  \makeatletter%
  \providecommand\color[2][]{%
    \errmessage{(Inkscape) Color is used for the text in Inkscape, but the package 'color.sty' is not loaded}%
    \renewcommand\color[2][]{}%
  }%
  \providecommand\transparent[1]{%
    \errmessage{(Inkscape) Transparency is used (non-zero) for the text in Inkscape, but the package 'transparent.sty' is not loaded}%
    \renewcommand\transparent[1]{}%
  }%
  \providecommand\rotatebox[2]{#2}%
  \ifx\svgwidth\undefined%
    \setlength{\unitlength}{529.20141602bp}%
    \ifx\svgscale\undefined%
      \relax%
    \else%
      \setlength{\unitlength}{\unitlength * \real{\svgscale}}%
    \fi%
  \else%
    \setlength{\unitlength}{\svgwidth}%
  \fi%
  \global\let\svgwidth\undefined%
  \global\let\svgscale\undefined%
  \makeatother%
  \begin{picture}(1,0.46538915)%
    \put(0,0){\includegraphics[width=\unitlength]{Fig4-WP-2013-6.pdf}}%
    \put(0.57452547,0.32188282){\color[rgb]{0,0,0}\makebox(0,0)[lb]{\smash{$\ell^{1/2}_{\alpha}=\varepsilon$}}}%
    \put(0.57452547,0.44411741){\color[rgb]{0,0,0}\makebox(0,0)[lb]{\smash{$\ell^{1/2}_{\alpha}=0$}}}%
    \put(0.16444607,0.31471797){\color[rgb]{0,0,0}\makebox(0,0)[lb]{\smash{$p$}}}%
    \put(0.57452547,0.20734953){\color[rgb]{0,0,0}\makebox(0,0)[lb]{\smash{$\ell^{1/2}_{\alpha}=2 \,\varepsilon$}}}%
    \put(0.22022827,0.29221945){\color[rgb]{0,0,0}\makebox(0,0)[lb]{\smash{$v$}}}%
    \put(0.50502059,0.24540805){\color[rgb]{0,0,0}\makebox(0,0)[lb]{\smash{$\gamma_v(t)$}}}%
    \put(0.68758144,0.26704771){\color[rgb]{0,0,0}\makebox(0,0)[lb]{\smash{$t\sim1/\varepsilon$}}}%
    \put(0.378958,0.09332011){\color[rgb]{0,0,0}\makebox(0,0)[lb]{\smash{$U$}}}%
  \end{picture}%
\endgroup%